\documentclass{amsart}
\usepackage{amssymb}
\usepackage{amsfonts}
\usepackage{amsmath}
\usepackage[toc,page]{appendix}
\usepackage{graphicx}

\newtheorem{theorem}{Theorem}[section]

\newtheorem{axiom}{Axiom}

\newtheorem{corollary}[theorem]{Corollary}

\newtheorem{definition}[theorem]{Definition}

\newtheorem{lemma}[theorem]{Lemma}

\newtheorem{proposition}[theorem]{Proposition}
\newtheorem{remark}[theorem]{Remark}

\input{tcilatex}

\begin{document}
\title{Characterization of distributivity in a solid}
\author{Bruno Dinis}
\address[B. Dinis]{Departamento de Matem\'{a}tica, Faculdade de Ci\^{e}ncias
da Universidade de Lisboa, Portugal.}
\email{bmdinis@fc.ul.pt}
\author{Imme van den Berg}
\address[I.P. van den Berg]{Departamento de Matem\'{a}tica, Universidade de 
\'{E}vora, Portugal}
\email{ivdb@uevora.pt}
\thanks{The first author acknowledges the support of the Funda\c{c}\~{a}o
para a Ci\^{e}ncia e a Tecnologia, Portugal [grant SFRH/BPD/97436/2013]}
\date{}

\begin{abstract}
We give a characterization of the validity of the distributive law in a
solid. There exists equivalence between the characterization and the
modified axiom of distibutivity valid in a solid.

\bigskip


\bigskip

\end{abstract}

\maketitle

\section{Introduction}

Solids arise is as extensions of fields \cite{dinisberg1}, typically
non-archimedean fields or the nonstandard reals \cite{koudjetithese}\cite%
{koudjetivandenberg}, in the form of cosets with respect to convex
subgroups. Such convex subgroups may be seen as orders of magnitude and are
called \emph{magnitudes}. In solids the laws of addition and multiplication
are more those of completely regular semigroups \cite{Howie}\cite{Petrich}
than of proper groups. Also the distributive law is not valid in general,
but there exists an adapted form of distributivity, introducing a correction
term in the form of a magnitude. In this paper we characterize the validity
of the ordinary distributive law in a solid (Theorem \ref{dist
chacterization}). Let $x,y,z$ be arbitrary elements in a solid. The
conditions of the characterization given in this paper roughly state that in
order for distributivity to fail the factor $x$ should be more imprecise
than the terms $y$ and $z$, and these terms should be almost opposite.
Special cases where distributivity does hold include magnitudes, elements of
the same sign, and precise elements (elements with minimal magnitude).

The equality expressed by the adapted distributivity axiom of \cite%
{dinisberg1} (Axiom \ref{Axiom distributivity} of the Appendix) and the
characterization of distributivity by Theorem \ref{dist chacterization} are
shown to be equivalent.

The results require a thorough investigation of the properties of magnitudes
and precise elements. This is done in Section \ref{Algebraic properties of
magnitudes} and \ref{Section precise elements}. In Subsection \ref%
{Subsection conditions distributivity} we state necessary and sufficient
conditions for distributivity to hold. In Subsection \ref{Subsection
equivalent form} we prove that the equivalency of these conditions to proper
distributivity is equivalent to the distributivity law with correction term,
as given by Axiom \ref{Axiom distributivity}.

For notation and terminology we refer to \cite{dinisberg1}. For the sake of
reference a complete list of axioms is given in the Appendix. We recall
that, given an element $x$ of a solid, its individualized neutral element $e$
(see Axiom \ref{assemblyneut}) as such is unique, and has the functional
notation $e=e(x)$. In the same way the symmetric element $s$ of Axiom \ref%
{assemblyneut} is denoted by $s(x)\equiv -x$, the individualized unity $u$
of Axiom \ref{axiom neut mult} is denoted by $u(x)$ and the multiplicative
inverse $d$ of Axiom \ref{axiom sym mult} is denoted by $d(x)\equiv
x^{-1}\equiv /x$.

\section{Algebraic properties of magnitudes\label{Algebraic properties of
magnitudes}}

In this section we study algebraic properties of magnitudes in an assembly.
Assemblies were introduced in \cite{dinisberg}. The results in this section
will be used to prove the characterization of distributivity in a solid.

\subsection{Neutral and symmetric elements}

We verify some elementary properties of magnitudes and symmetrical elements.
Part of it are generalizations of the usual properties of neutral and
symmetric elements and others deal with their functional representation.

We start by recalling some results on magnitudes of \cite{dinisberg}.

\begin{theorem}[{\protect \cite[Thm 4.11]{dinisberg}}]
\label{Cancelation law} \emph{(Cancellation law)} Let $A$ be an assembly and
let $x,y,z\in A$. Then $x+y=x+z$ if and only if $e(x)+y=e(x)+z.$
\end{theorem}

\begin{proposition}[{\protect \cite[Thm 4.12]{dinisberg}}]
\label{theoremassembly}Let $A$ be an assembly. Then for all $x,y\in A$
\end{proposition}

\begin{enumerate}
\item \label{assembly1}(Idempotency for addition) $e(x)+e(x)=e(x)$.

\item \label{assembly2}(Linearity of $e$) $e(x+y)=e(x)+e(y)$.

\item \label{assembly0}(Absorption) $e(x)+e(y)=e(x)$ or $e(x)+e(y)=e(y)$.

\item \label{assembly3}(Idempotency for composition)\emph{\ }Let $A$ be an
assembly and let $x\in A$. Then $e(e(x))=e(x).$
\end{enumerate}

As a consequence a magnitude can only be the magnitude of itself.

\begin{theorem}
\label{assembly7}\emph{(Representation)} Let $A$ be an assembly and let $%
x,y\in A$. If $x=e\left( y\right) $ then $x=e\left( x\right) $.
\end{theorem}

\begin{proof}
Suppose $y\in A$ is such that $x=e(y)$. Then $e(x)=e(e(y))=e(y)=x$ by
Proposition \ref{theoremassembly}.\ref{assembly3}.
\end{proof}

As regards to the symmetric function, it is easy to see that it is
injective. We verify now that it is linear with respect to addition and has
the symmetry property, meaning that the inverse of the inverse of a given
element is the element itself. We also show that the composition of the
inverse function with the neutral function is equal to the neutral function.

\begin{proposition}
\label{prop sym}Let $A$ be an assembly and let $x\in A$. Then
\end{proposition}

\begin{enumerate}
\item \label{assembly4} $-(-x)=x$.

\item \label{assembly5}\emph{\ }$-(x+y)=-x-y$.

\item \label{assembly6} $e(-x)=-e(x)=e(x)$.
\end{enumerate}

\begin{proof}
\ref{assembly4}. By Axiom \ref{assemblysim} one has $%
e(-(-x))=e(-x)=e(x)=-x+x $. Hence $%
-(-x)=-(-x)+e(-(-x))=-(-x)-x+x=e(-x)+x=e(x)+x=x.$

\ref{assembly5}. By Proposition \ref{theoremassembly}.\ref{assembly2} and
Axiom \ref{assemblysim} one has $%
-(x+y)+x+y=e(x+y)=e(x)+e(y)=-x+x-y+y=-x-y+x+y$. Then by the cancellation law 
$-(x+y)+e(x+y)=-x-y+e(x+y)$. Again using the linearity of $e$ and Axiom \ref%
{assemblysim} one obtains 
\begin{align*}
-(x+y)+e(-(x+y))& =-x-y+e(-x)+e(-y) \\
& =-x+e(-x)-y+e(-y).
\end{align*}%
Hence $-(x+y)=-x-y.$

\ref{assembly6}. By Axiom \ref{assemblysim} one only has to show that $%
-e(x)=e(x)$. Using Proposition \ref{theoremassembly}.\ref{assembly4} and
Proposition \ref{theoremassembly}.\ref{assembly5} one derives $%
e(x)=-x+x=-x-(-x)=-(x-x)=-e(x)$.
\end{proof}

\subsection{Order and the magnitudes}

\begin{definition}
We say that an assembly $A$ is an \emph{ordered assembly} if it satisfies
the axioms of order \ref{(OA)reflex}-\ref{Axiom Amplification}.
\end{definition}

As far as magnitudes are concerned, the order relation can be defined in
terms of addition and corresponds to the natural order relation in
semigroups \cite{Howie}\cite{Petrich}.

\begin{proposition}
\label{e(x)maiore(y)iff}Let $A$ be an ordered assembly. Let $x,y\in A$. Then 
$e(x)+e(y)=e(x)$ if and only if $e(y)\leq e(x)$.
\end{proposition}

\begin{proof}
By Axiom \ref{Axiom e(x)maiory} we only need to prove the necessary part.
Suppose that $e(y)\leq e(x)$. Now $e(x)+e(y)=e(x)$ or $e(x)+e(y)=e(y)$ by
Proposition \ref{theoremassembly}.\ref{assembly2}. If $e(x)+e(y)=e(x)$,
there is nothing to prove. If $e(x)+e(y)=e(y)$, by Axiom \ref{Axiom
e(x)maiory}, $e(x)\leq e(y)$. But then, by antisymmetry, $e(x)=e(y)$. Hence $%
e(x)+e(y)=e(x)+e(x)=e(x)$.
\end{proof}

We say that $x$ is \emph{positive} if $e\left( x\right) \leq x$ and\emph{\
negative if }$x<e(x)$. With these notions it is possible to define an
absolute value.

\begin{definition}
\label{Absolutevalue}Let $A$ be an ordered assembly and let $x\in A$. The 
\emph{absolute value of }$x$ is defined as%
\begin{equation*}
\left \vert x\right \vert \equiv \left \{ 
\begin{array}{c}
x,\text{ if }e(x)\leq x \\ 
-x\text{, if }x<e(x).%
\end{array}
\right.
\end{equation*}
\end{definition}

It follows from compatibility with addition that $e\left( x\right) \leq x$
if and only if $-x\leq e\left( x\right) $. We show that the sum of two
positive elements is also positive and that every element which is larger
than or equal to a positive element is also positive. Then we prove that
Proposition \ref{e(x)maiore(y)iff} may be extended to any positive element.
We finish with some strict inequalities.

\begin{proposition}
\label{order sum}Let $A$ be an ordered assembly and let $x,y\in A$.
\end{proposition}

\begin{enumerate}
\item \label{xmaiore(x)7}If $x$ and $y$ are both positive then $x+y$ is also
positive.

\item \label{y<=x y positive then x positive}If $e\left( y\right) \leq y\leq
x$ then $e\left( x\right) \leq x$.
\end{enumerate}

\begin{proof}
\ref{xmaiore(x)7}. If $e(x)\leq x$ and $e(y)\leq y$, then $%
e(x+y)=e(x)+e(y)\leq x+y$.

\ref{y<=x y positive then x positive}. Assume that $e\left( y\right) \leq
y\leq x$. If $e\left( y\right) \leq e\left( x\right) $ then $e\left(
x\right) +e\left( y\right) =e\left( x\right) $ by Proposition \ref%
{e(x)maiore(y)iff}. By compatibility with addition $e\left( x\right)
=e\left( x\right) +e\left( y\right) \leq e\left( x\right) +x=x$. If $e\left(
x\right) \leq e\left( y\right) $, by transitivity $e\left( x\right) \leq x$.
\end{proof}

The next theorem states that a positive number leaves a magnitude invariant
if and only if it is smaller than this magnitude.

\begin{theorem}
\label{emaiory e+y=e}Let $A$ be an ordered assembly and let $x,y\in A$. If $%
y $ is positive then $y\leq e(x)$ if and only if $e(x)+y=e(x).$
\end{theorem}

\begin{proof}
By Axiom \ref{Axiom e(x)maiory} we only need to prove the sufficiency.
Assume that $y$ is positive and that $y\leq e(x)$. Then by transitivity $%
e(y)\leq e(x)$ and $e(x)+e(y)=e(x)$ by Proposition \ref{e(x)maiore(y)iff}.
By compatibility with addition $e(x)=e(x)+e(y)\leq e(x)+y\leq e(x)+e(x)=e(x)$%
. Hence $e(x)=e(x)+y$.
\end{proof}

\begin{proposition}
\label{lema_comp_multiplication}Let $A$ be an ordered assembly and let $%
x,y,z\in A$.
\end{proposition}

\begin{enumerate}
\item \label{xneg ypos x<y}If $x<e\left( x\right) $ and $e\left( y\right)
<y, $ then $x<y$.

\item \label{x<e(x) => x<e(y)}If $x<e\left( x\right) $, then $x<e\left(
y\right) $.

\item \label{x<y+e(z) => x<y}If $x<y+e\left( z\right) $ and $e\left(
z\right) <e\left( x\right) $, then $x<y$.
\end{enumerate}

\begin{proof}
\ref{xneg ypos x<y}. Assume $x<e\left( x\right) $ and $e\left( y\right) <y$.
If $e\left( x\right) \leq e\left( y\right) $ the result follows by
transitivity. If $e\left( y\right) <e\left( x\right) $ then $e\left(
x\right) =e\left( y\right) +e\left( x\right) $ by Proposition \ref%
{e(x)maiore(y)iff}. Suppose, towards a contradiction, that $y\leq x$. Then $%
e\left( x\right) =e\left( y\right) +e\left( x\right) \leq y+e\left( x\right)
\leq x+e\left( x\right) =x$, a contradiction.

\ref{x<e(x) => x<e(y)}. Assume $x<e\left( x\right) $. We suppose towards a
contradiction $e\left( y\right) \leq x$. Then $e\left( y\right) <e\left(
x\right) $ and $e\left( x\right) =e\left( y\right) +e\left( x\right) $ by
Proposition \ref{e(x)maiore(y)iff}. Hence $e\left( x\right) =e\left(
y\right) +e\left( x\right) \leq x+e\left( x\right) =x$, a contradiction.

\ref{x<y+e(z) => x<y}. Assume $x<y+e\left( z\right) $ and $e\left( z\right)
<e\left( x\right) $. By Proposition \ref{e(x)maiore(y)iff}, one has $e\left(
z\right) +e\left( x\right) =e\left( x\right) $. Suppose towards a
contradiction that $y\leq x$. Then $y+e\left( z\right) \leq x+e\left(
z\right) =x$, a contradiction.
\end{proof}

\subsection{Magnitudes and the product\label{subsection product}}

We denote by $S^{\ast }$ the set of all elements of $S$ which are not
magnitudes, i.e. $S^{\ast }=\left \{ x\in S|\text{ }x\neq e\left( x\right)
\right \} $.

The following lemma shows that unity elements of zeroless elements are
zeroless, implying that magnitudes and unities in a solid are distinct.

\begin{lemma}
\label{prop dif}Let $S$ be a solid and let $x,y\in S^{\ast }$
\end{lemma}

\begin{enumerate}
\item \label{udife(u)}$u(x)\neq e(u(x))$.

\item \label{udife}$u(x)\neq e(y)$.
\end{enumerate}

\begin{proof}
\ref{udife(u)}. Suppose that $u(x)=e(u(x))$. Then $x=xu(x)=xe(u(x))$. By
Axiom \ref{Axiom escala} there is $z$ such that $x=e(z)$ and by Theorem \ref%
{assembly7}, $x=e(x)$. Hence, if $x\neq e(x)$, then $u(x)\neq e(u(x))$.

\ref{udife}. Suppose towards a contradiction that $u\left( x\right) =e\left(
y\right) $ for some $y\in A$. Then $u\left( x\right) =e\left( u\left(
x\right) \right) $ by Theorem \ref{assembly7}, in contradiction with Part %
\ref{udife(u)}.
\end{proof}

As a consequence of the previous lemma multiplicative inverses of zeroless
elements are zeroless. Indeed, if $x^{-1}\neq e\left( x^{-1}\right) $, one
must have $u(x)=xx^{-1}=xe\left( x^{-1}\right) $. Then $u(x)$ would be a
magnitude by Axiom \ref{Axiom escala}, a contradiction.

Let $S$ be a solid and let $x,y\in S^{\ast }$. Concerning the magnitude of
the product, one should expect that%
\begin{align}
e(xy)& =e((x+e(x))(y+e(y))  \label{=e(x)y+e(y)x+e(x)e(y).} \\
& =e(xy+e(x)y+e(y)x+e(x)e(y))  \notag \\
& =e(x)y+e(y)x+e(x)e(y).  \notag
\end{align}%
Axiom \ref{Axiom e(xy)=e(x)y+e(y)x} states that we may neglect the last
term. In the next subsection, using the order relation, we show that the
term $e(x)e(y)$ is less indeed than both $e(x)y$ and $e(y)x$.

Here we present some usefull properties of multiplication by magnitudes
obtained by mere algebraic methods.

If one of the factors of the product is a magnitude, we have the following
simplification.

\begin{proposition}
\label{e(x)y=e(xy)}Let $S$ be a solid and let $x,y\in S$. If $x=e(x)$ then $%
e(xy)=e(x)y$.
\end{proposition}

\begin{proof}
By Axiom \ref{Axiom escala}, $e\left( x\right) y$ is a magnitude. Then $%
e\left( xy\right) =e\left( e\left( x\right) y\right) =e\left( x\right) y$.
\end{proof}

Whenever one multiplies a given element by a magnitude, the sign of that
element can be neglected.

\begin{proposition}
\label{es=ex}Let $S$ be a solid and let $x,y\in S$. Then $e(y)(-x)=e(y)x$.
\end{proposition}

\begin{proof}
By Axiom \ref{Axiom s(xy)=s(x)y} and Proposition \ref{prop sym}.\ref%
{assembly6} one has $e(y)(-x)=-\left( e\left( y\right) x\right)
=(-e(y))x=e(y)x$.
\end{proof}

In integral domains the product of two non-zero elements is always non-zero.
In solids the product of two zeroless elements is always zeroless.

\begin{theorem}
\label{zerodiv}Let $S$ be a solid and let $x,y\in S$. Then $xy=e\left(
xy\right) $ if and only if $x=e\left( x\right) $ or $y=e\left( y\right) $.
\end{theorem}

\begin{proof}
Suppose firstly that $x=e\left( x\right) $ or $y=e\left( y\right) $. If $%
x=e\left( x\right) $, then $xy=e\left( x\right) y$ and by Axiom \ref{Axiom
escala} there is $z$ such that $xy=e\left( z\right) $. By Theorem \ref%
{assembly7} we conclude that $xy=e\left( xy\right) $. If $y=e\left( y\right) 
$ the proof is analogous.

Secondly, suppose towards a contradiction that $xy=e\left( xy\right) $,
while $x\neq e\left( x\right) $ and $y\neq e\left( y\right) $. Without loss
of generality we may assume that $u\left( x\right) u\left( y\right) =u\left(
x\right) $. Then%
\begin{equation*}
u\left( x\right) =u\left( x\right) u\left( y\right) =xx^{-1}yy^{-1}=e\left(
xy\right) x^{-1}y^{-1}\text{.}
\end{equation*}%
By Axiom \ref{Axiom escala} there is $z$ such that $e\left( z\right)
=e\left( xy\right) x^{-1}y^{-1}$, in contradicion with Lemma \ref{prop dif}.
Hence if $xy=e\left( xy\right) $, then $x=e\left( x\right) $ or $y=e\left(
y\right) $.
\end{proof}

As a corollary we obtain that if an element is not a magnitude, its square
(element multiplied by itself) is also not a magnitude.

\begin{corollary}
\label{xquadrado dif e}Let $S$ be a solid and let $x\in S^{\ast}$, then $%
x^{2}\neq e\left( x^{2}\right) $.
\end{corollary}

In rings one has that $1\cdot0=0$. Next theorem generalizes this to solids.
A further generalization will be given in Theorem \ref{e(x)u(y)=e(x)}\
below. As a consequence we obtain an expression for the magnitude of the
inverse.

\begin{theorem}
\label{equivalences}Let $S$ be a solid. Suppose $x\in S^{\ast}.$ Then

\begin{enumerate}
\item \label{eu=xe(u)}$u(x)e(x)=xe(u(x))=e(x).$

\item \label{e(d(x))=edd}$e(x^{-1})=e(x)x^{-1}x^{-1}$.
\end{enumerate}
\end{theorem}

\begin{proof}
\ref{eu=xe(u)}. Suppose that $x\neq e(x)$. Then $e\left( x\right) =e\left(
xu\left( x\right) \right) =e\left( x\right) u\left( x\right) +xe\left(
u\left( x\right) \right) $ by Axiom \ref{axiom neut mult} and Axiom \ref%
{Axiom e(xy)=e(x)y+e(y)x}. Hence applying Axiom \ref{e(u(x))=e(x)d(x)} 
\begin{equation*}
e(x)=e(x)u(x)+xe(x)x^{-1}=e(x)u(x)+e(x)u(x)=e(x)u(x).
\end{equation*}%
This implies that 
\begin{equation*}
xe(u(x))=xe(x)x^{-1}=e(x)u(x)=e(x).
\end{equation*}%
\bigskip \ref{e(d(x))=edd}. Note that by Axiom \ref{e(u(x))=e(x)d(x)} 
\begin{equation*}
e(x^{-1})x=e(x^{-1})(x^{-1})^{-1}=e(u(x^{-1}))=e(u(x))=e(x)x^{-1}.
\end{equation*}%
Hence%
\begin{equation*}
e(x^{-1})=e(x^{-1})u(x^{-1})=e(x^{-1})u(x)=e(x^{-1})xx^{-1}=e(x)x^{-1}x^{-1}.
\end{equation*}
\end{proof}

\subsection{Order and the product\label{subsection preservation properties}}

We first investigate how the order behaves under multiplication by
magnitudes. As a result we obtain that the product of two positive numbers
is still positive, implying that squares are positive and a fortiori the
unity. Inequalities are reversed when multiplying by a negative element upon
the application of a correction term. The rules of the sign of the inverse
are as in rings.

Next proposition states that the order relation is preserved under scaling,
and Lemma \ref{xemaioree} and Proposition \ref{emaioree} indicate that the
neglection of the term $e(x)e(y)$ in the formula for the magnitude of the
product of Axiom \ref{Axiom e(xy)=e(x)y+e(y)x} is justified, as anounced in
the previous subsection.

\begin{proposition}
\label{xe(y)=xe(z)}Let $S$ be a solid and let $x,y,z\in S$. If $e(y)\leq
e(z) $ then $xe(y)\leq xe(z)$.
\end{proposition}

\begin{proof}
If $e\left( x\right) =x$ the result follows by Axiom \ref{Axiom
Amplification}. If $e(x)<x$, by Axiom \ref{compat mult}\ one has $xe(y)\leq
xe(z)$. If $x<e(x)$, then\ $e(-x)=e\left( x\right) <-x$. Hence\ $%
-xe(y)\leq-xe(z)$ by compatibility with multiplication. By Proposition \ref%
{es=ex}\ one has $xe(y)\leq xe(z)$.
\end{proof}

\begin{lemma}
\label{xemaioree}Let $S$ be a solid and let $x,y\in S$. Then $e(x)e(y)\leq
xe(y)$.
\end{lemma}

\begin{proof}
If $e\left( x\right) =x$ there is nothing to show. If $e(x)<x$, the result
follows by compatibility with multiplication\ because\ $e(y)\leq e(y)$. If $%
x<e(x)$, then $e(x)<-x$. Then $e(x)e(y)\leq-xe(y)=xe(y)$ by Proposition \ref%
{es=ex} and Axiom \ref{Axiom Amplification}.
\end{proof}

\begin{proposition}
\label{emaioree}Let $S$ be a solid and let $x,y\in S$. Then $e(x)e(y)\leq
e(xy)$.
\end{proposition}

\begin{proof}
By Lemma \ref{xemaioree} it holds that $e(x)e(y)\leq xe(y)$ and $%
e(x)e(y)\leq ye(x)$. Hence using Axiom \ref{Axiom e(xy)=e(x)y+e(y)x},%
\begin{equation*}
e(x)e(y)\leq xe(y)+ye(x)=e(xy)\text{.}
\end{equation*}
\end{proof}

We are now able to prove that products of two positive elements and squares
are positive.

\begin{theorem}
\label{mult positivos}Let $S$ be a solid and let $x,y\in S$. If $x$ and $y$
are both positive then $e\left( xy\right) \leq xy$.
\end{theorem}

\begin{proof}
Suppose $x$ and $y$ are both positive. If $y=e\left( y\right) $ or $%
x=e\left( x\right) $ by Axiom \ref{Axiom escala} and Theorem \ref{assembly7}
one has $e\left( xy\right) =xy$. If $e\left( x\right) <x$ and $e\left(
y\right) <y$, by Axiom \ref{compat mult}, $xe(y)\leq xy$ and $ye(x)\leq xy$.
Then, using Axiom \ref{Axiom e(xy)=e(x)y+e(y)x},%
\begin{equation*}
e(xy)=xe(y)+ye(x)\leq xy+xy\text{.}
\end{equation*}%
By adding $-(xy)$ to both sides of the equation one obtains $-xy\leq xy$.
Then $e(xy)\leq xy$.
\end{proof}

\begin{corollary}
\label{Cor positive}Let $S$ be a solid and let $x\in S^{\ast }$. Then
\end{corollary}

\begin{enumerate}
\item \label{quadrado positivo} $e(x^{2})\leq x^{2}$. Moreover, equality
holds if and only if $x=e\left( x\right) $.

\item \label{u is positive} $e\left( u\left( x\right) \right) <u\left(
x\right) $.
\end{enumerate}

\begin{proof}
\ref{quadrado positivo}. We show firstly that $e(x^{2})\leq x^{2}$. If $x$
is positive, the result follows from Theorem \ref{mult positivos}. If $x$ is
negative, then $-x$ is positive. Hence 
\begin{equation*}
e(x^{2})=e((-x)^{2})\leq(-x)^{2}=x^{2}.
\end{equation*}

Secondly, if $x=e\left( x\right) $, Proposition \ref{emaioree} implies that $%
x^{2}=x^{2}+e\left( x^{2}\right) =e(x)^{2}+e\left( x^{2}\right) =e\left(
x^{2}\right) $. If $x\neq e\left( x\right) $, by Corollary \ref{xquadrado
dif e} one has $x^{2}\neq e\left( x^{2}\right) $. Hence equality holds if
and only if $x=e\left( x\right) $.

\ref{u is positive}. By Part \ref{quadrado positivo} because $u\left(
x\right) u\left( x\right) =u\left( x\right) $.
\end{proof}

In ordered rings inequalities are reversed upon multiplying by a negative
element. This remains true for zeroless elements.

\begin{proposition}
\label{lema_comp_mult}Let $S$ be a solid and let $x,y\in S$. If $e\left(
x\right) <x$ and $y<e\left( y\right) $ then $xy<e\left( xy\right) $.
\end{proposition}

\begin{proof}
Suppose $e\left( x\right) <x$ and $y<e\left( y\right) $. Then $e\left(
y\right) <-y$, hence $e\left( x\right) e\left( y\right) \leq x\left(
-y\right) $. Then by Proposition \ref{emaioree} and compatibility with
addition 
\begin{equation*}
xy=xy+e\left( x\right) e\left( y\right) \leq xy+x\left( -y\right)
=xy-xy=e\left( xy\right) .
\end{equation*}%
Because $x\neq e\left( x\right) $ and $y\neq e\left( y\right) $ one has $%
xy\neq e\left( xy\right) $, by Theorem \ref{zerodiv}. Hence $xy<e\left(
xy\right) $.
\end{proof}

In the general case one must take into account the magnitudes of the
products of the elements.

\begin{proposition}
Let $S$ be a solid and let $x,y,z\in S$. Suppose that $y\leq z$. If $%
x<e\left( x\right) $ then $xz+e(xy)\leq xy+e(xz)$.
\end{proposition}

\begin{proof}
Suppose that $y\leq z$. If $x<e\left( x\right) $, then $e(-x)=e\left(
x\right) <-x$. By compatibility with multiplication $\left( -x\right) y\leq
\left( -x\right) z$. Then $-\left( xy\right) \leq -\left( xz\right) $ by
Axiom \ref{Axiom s(xy)=s(x)y}. Applying Axiom \ref{(OA)compoper} twice, we
see that $xz+e(xy)\leq xy+e(xz)$.
\end{proof}

Within ordered rings with unity it holds that $0<1$, the inverse for
multiplication of a positive element is positive, and the inverse of an
element larger than $1$ is smaller than $1$. We adapt these properties to
solids.

\begin{proposition}
\label{1/x is positive}Let $S$ be a solid and let $x\in S^{\ast}$. Then $%
e\left( x\right) <x$ if and only if $e\left( x^{-1}\right) <x^{-1}$.
\end{proposition}

\begin{proof}
Assume that $e\left( x\right) <x$. Because $x\neq e\left( x\right) $, by
Lemma \ref{prop dif}.\ref{udife} one has $x^{-1}\neq e\left( x^{-1}\right) $%
. By Theorem \ref{zerodiv} one has $x^{-1}x^{-1}\neq e\left(
x^{-1}x^{-1}\right) $. Then by compatibility with multiplication and
Corollary \ref{Cor positive}.\ref{quadrado positivo}%
\begin{equation*}
xe\left( x^{-1}x^{-1}\right) \leq xx^{-1}x^{-1}=u\left( x\right)
x^{-1}=x^{-1}.
\end{equation*}%
Then by Theorem \ref{equivalences}.\ref{eu=xe(u)} and Axiom \ref{Axiom
e(xy)=e(x)y+e(y)x} 
\begin{align*}
e\left( x^{-1}\right) & =u\left( x^{-1}\right) e\left( x^{-1}\right) \\
& =u\left( x\right) e\left( x^{-1}\right) \\
& =xe\left( x^{-1}\right) x^{-1} \\
& =x\left( e\left( x^{-1}\right) x^{-1}+x^{-1}e\left( x^{-1}\right) \right)
\\
& =xe\left( x^{-1}x^{-1}\right) \leq x^{-1}.
\end{align*}%
Hence $e\left( x^{-1}\right) <x^{-1}$.

Conversely, assume that $e\left( x^{-1}\right) <x^{-1}$. By the above $%
e(x)=e((x^{-1})^{-1})<(x^{-1})^{-1}=x$.
\end{proof}

\begin{corollary}
\label{Cor u>e}Let $S$ be a solid and let $x\in S^{\ast}$. If $u\left(
x\right) \leq x$ then $e\left( x\right) <x$.
\end{corollary}

\begin{proof}
By Proposition \ref{order sum}.\ref{y<=x y positive then x positive}, with $%
y=u\left( x\right) $, which is positive by Corollary \ref{Cor positive}.\ref%
{u is positive}.
\end{proof}

\begin{proposition}
Let $S$ be a solid and let $x\in S^{\ast}$. Then

\begin{enumerate}
\item \label{umaiord}If $u(x)\leq x$, then $x^{-1}\leq u(x)$.

\item \label{dmaioru}If $e(x)<x$ and $x\leq u(x)$, then $u(x)\leq x^{-1}$.
\end{enumerate}
\end{proposition}

\begin{proof}
\ref{umaiord}. Suppose that $u(x)\leq x$. Then by Corollary \ref{Cor u>e}
one has $e\left( x\right) <x$ and then $e(x^{-1})<x^{-1}$ by Proposition \ref%
{1/x is positive}. Hence Axiom \ref{compat mult} implies that $%
x^{-1}=x^{-1}u(x)\leq x^{-1}x=u\left( x\right) $.

\ref{dmaioru}. Suppose that $e(x)<x$ and $x\leq u(x)$. Then $%
e(x^{-1})<x^{-1} $ by Proposition \ref{1/x is positive}. Again Axiom \ref%
{compat mult} implies that $u\left( x\right) =xx^{-1}\leq
u(x)x^{-1}=u(x^{-1})x^{-1}=x^{-1} $.
\end{proof}

Finally we show that if a positive element is larger than its inverse, its
magnitude must be larger than the magnitude of the inverse.

\begin{proposition}
\label{eumaiore(d)}Let $S$ be a solid and let $x\in S^{\ast}$. If $e\left(
x^{-1}\right) <x^{-1}\leq x$ then $e\left( x^{-1}\right) \leq e\left(
x\right) $.
\end{proposition}

\begin{proof}
Suppose $e\left( x^{-1}\right) <x^{-1}\leq x$. Then $x^{-1}x^{-1}\leq
xx^{-1}=u\left( x\right) $, by Axiom \ref{compat mult}. By Theorem \ref%
{equivalences} and Axiom \ref{Axiom Amplification} one has 
\begin{equation*}
e\left( x^{-1}\right) =x^{-1}x^{-1}e\left( x\right) \leq u(x)e(x)=e\left(
x\right) \text{.}
\end{equation*}
\end{proof}

\subsection{Distributivity with magnitudes\label{Subsection elementary cases}%
}

In this subsection we prove distributivity in some special cases, without
using Axiom \ref{Axiom distributivity}. We show that the distributive
property $x(y+z)=xy+xz$ holds in case $y$ and $z$ are both magnitudes, and
in case one of those is a magnitude, less than or equal to the magnitude of
the remaining element. In particular one always has $x(y+e(y))=xy+xe(y)$.

\begin{proposition}
\label{dist neutrices}Let $S$ be a solid and let $x,y,z\in S$. Then $x\left(
e(y)+e(z)\right) =xe(y)+xe(z)$.
\end{proposition}

\begin{proof}
We may suppose without loss of generality that $e(z)\leq e(y)$. Then $%
e(y)+e(z)=e(y)$ by Proposition \ref{e(x)maiore(y)iff}. By Proposition \ref%
{xe(y)=xe(z)}\ one has $e(z)x\leq e(y)x$. Then Axiom \ref{Axiom escala} and
Proposition \ref{e(x)maiore(y)iff} imply that $e(y)x+e(z)x=e(y)x$.

Hence%
\begin{equation*}
x\left( e(y)+e(z)\right) =xe(y)=xe(y)+xe(z).
\end{equation*}
\end{proof}

As a consequence of the previous proposition distributivity holds for the
magnitudes.

\begin{corollary}
\label{Cor dist magnitudes}Let $S$ be a solid and let $x,y,z\in S$. Then $%
e(x)\left( e(y)+e(z)\right) =e(x)e(y)+e(x)e(z)$.
\end{corollary}

\begin{proposition}
\label{dist neutrices dif}Let $S$ be a solid and let $x,y,z\in A$. If $%
e(z)\leq e(y)$, then $x(y+e(z))=xy+xe(z)$.
\end{proposition}

\begin{proof}
Assume that $e(z)\leq e(y)$. Then by Proposition \ref{e(x)maiore(y)iff} one
has $x(y+e(z))=x(y+e(y)+e(z))=xy$ . Now $xe\left( z\right) \leq xe\left(
y\right) $ by Proposition \ref{xe(y)=xe(z)} and $xe\left( y\right) +xe\left(
z\right) =xe\left( y\right) $ by Proposition \ref{e(x)maiore(y)iff}. Then
Axiom \ref{Axiom e(xy)=e(x)y+e(y)x} implies that 
\begin{align*}
xy+xe(z) & =xy+e(xy)+xe(z) \\
& =xy+e(x)y+xe(y)+xe(z) \\
& =xy+e(x)y+xe(y)=xy+e(xy)=xy\text{.}
\end{align*}
Hence $x(y+e(z))=xy+xe(z)$.
\end{proof}

\begin{corollary}
\label{Cor.x(y+e(y))=xy+xe(y)}Let $S$ be a solid and let $x,y\in S$. Then $%
xy=x(y+e(y))=xy+xe(y)$.
\end{corollary}

\section{The field of precise elements\label{Section precise elements}}

\begin{definition}
Let $S$ be a solid. An element $x\in S$ is said to be \emph{precise} if $%
e\left( x\right) =m$.
\end{definition}

The elements $m$ and $u$ given by Axioms \ref{Axiom neut min} and \ref{Axiom
neut mult} in solids are unique. We call them zero and one respectively. We
prove that zero and one have the expected properties concerning the order
relation and multiplication. The precise elements are closed under addition,
multiplication and inversion, in fact constitute a field.

The proof of the fact that the elements $m$ and $u$ are unique is identical
as for groups and will be omitted.

\begin{proposition}
\label{unicid zero}Let $S$ be a solid. There is exactly one element $m$ and
exactly one element $u$ such that $x+m=x$ and $xu=x$ for all $x\in S$.
\end{proposition}

\begin{definition}
We call \emph{zero} the unique element $m$ such that $x+m=x$ for all $x\in S$%
, and it will be denoted by $0$. We call \emph{one} the unique element $u$
such that $xu=x$ for all $x\in S$, and it will be denoted by $1.$
\end{definition}

In the following proposition we show some elementary properties of the
elements $0$ and $1$.

\begin{proposition}
\label{prop zero}Let $S$ be a solid and $x,y,z\in S$. Then
\end{proposition}

\begin{enumerate}
\item \label{e(0)=0}$e\left( 0\right) =0$.

\item \label{e(1)=0}$e\left( 1\right) =0$.

\item \label{1=note(1)}$1\neq 0$.

\item \label{u(1)=1}$u\left( 1\right) =1$.

\item \label{-0=0}$-0=0$.

\item \label{1.0=0}$1\cdot0=0$.

\item \label{0<1}$0<1$.

\item \label{e(x)maiorzero}$0\leq e\left( x\right) $.

\item \label{xmaior0}If $0\leq x$ then $e\left( x\right) \leq x$.

\item \label{x+ymaior0}If $0\leq x$ and $0\leq y$, then $0\leq x+y$.

\item \label{x+menor0}If $x\leq 0$ and $y\leq 0$, then $x+y\leq 0$.
\end{enumerate}

\begin{proof}
\ref{e(0)=0}. By Axiom \ref{assemblyneut} and \ref{Axiom neut min} one has $%
0=0+e\left( 0\right) =e\left( 0\right) $.

\ref{e(1)=0}. Let $x\in S$. Then $x\left( 1+e\left( 1\right) \right) =x1=x$.
By Corollary \ref{Cor.x(y+e(y))=xy+xe(y)} 
\begin{equation*}
x\left( 1+e\left( 1\right) \right) =x1+xe\left( 1\right) =x+xe\left(
1\right) .
\end{equation*}%
Hence $x=x+xe\left( 1\right) $. Since $x$ is arbitrary, by Proposition \ref%
{unicid zero} one has $xe\left( 1\right) =0$. Putting $x=1$, one obtains $%
1e\left( 1\right) =0$. Since $1e\left( 1\right) =e\left( 1\right) $ by Axiom %
\ref{Axiom neut mult}, one derives that $e\left( 1\right) =0$.

\ref{1=note(1)}. Because $S$ is an solid, it has a zeroless element, say $x$%
. If $1=0$, then $x=x\cdot 1=x\cdot 0=x\cdot e\left( 0\right) $, so by Axiom %
\ref{Axiom escala} $x$ is a magnitude, a contradiction. Hence $1\neq 0$.

\ref{u(1)=1}. By Axiom \ref{axiom neut mult} and \ref{Axiom neut mult} one
has $1=1\cdot u\left( 1\right) =u\left( 1\right) $.

\ref{-0=0}. Using Part \ref{e(0)=0} and Proposition \ref{prop sym}.\ref%
{assembly6} one has $-0=-e\left( 0\right) =e\left( 0\right) =0$.

\ref{1.0=0}. By Part \ref{u(1)=1}, Part \ref{e(1)=0} and Theorem \ref%
{equivalences}.\ref{eu=xe(u)}.

\ref{0<1}. Corollary \ref{Cor positive}.\ref{u is positive}, Part \ref%
{e(1)=0} and Part \ref{u(1)=1} imply that $0=e(1)=e(u(1))<u(1)=1$.

\ref{e(x)maiorzero}. Directly from Axiom \ref{Axiom neut min} and Axiom \ref%
{Axiom e(x)maiory}.

\ref{xmaior0}. Suppose that $0\leq x$. Then $e\left( x\right) =0+e\left(
x\right) \leq x+e\left( x\right) =x$ by compatibility with addition.

\ref{x+ymaior0}. Directly from Axiom \ref{(OA)compoper}.

\ref{x+menor0}. Directly from Axiom \ref{(OA)compoper}.
\end{proof}

We prove that zero is the absorbing element for multiplication.

\begin{proposition}
\label{x0=0}Let $S$ be a solid. Then $x0=0$, for all $x\in S$.
\end{proposition}

\begin{proof}
By Proposition \ref{prop zero}.\ref{e(1)=0} and Corollary \ref%
{Cor.x(y+e(y))=xy+xe(y)} one has 
\begin{equation*}
x\left( 1+0\right) =x1+x0
\end{equation*}%
for all $x\in S$. Hence $x=x+x0$ for all $x\in S$. Then $x0=0$ for all $x\in
S$ by Proposition \ref{unicid zero}.
\end{proof}

By Proposition \ref{prop zero}.\ref{e(0)=0} and \ref{prop zero}.\ref{e(1)=0}
the elements $0$ and $1$ are examples of precise elements. We show that
precise elements are closed under elementary operations.

\begin{proposition}
\label{precise closed}Let $S$ be a solid and let $a,b\in S$ be precise. Then 
$a+b,-b$ and $ab$ are precise. If $a\neq 0$ then $u\left( a\right) $ and $%
a^{-1}$ are also precise.
\end{proposition}

\begin{proof}
Since $a$ and $b$ are precise, one has $e\left( a\right) =e\left( b\right)
=0 $. Then $e\left( -b\right) =e\left( b\right) =0$ and $e\left( a+b\right)
=e\left( a\right) +e\left( b\right) =0$. Hence $-b$ and $a+b$ are precise.
Also, by Axiom \ref{Axiom e(xy)=e(x)y+e(y)x} and Proposition \ref{x0=0}%
\begin{equation*}
e\left( ab\right) =e\left( a\right) b+ae\left( b\right) =0.
\end{equation*}%
Hence $ab$ is precise.

Suppose now that $a\neq 0$. Then by Axiom \ref{e(u(x))=e(x)d(x)}\ and
Proposition \ref{x0=0}%
\begin{equation*}
e\left( u\left( a\right) \right) =e\left( a\right) a^{-1}=0a^{-1}=0\text{.}
\end{equation*}%
Hence $u\left( a\right) $ is precise. Finally, by Theorem \ref{equivalences}.%
\ref{e(d(x))=edd} and Proposition \ref{x0=0}%
\begin{equation*}
e\left( a^{-1}\right) =e\left( a\right) a^{-1}a^{-1}=0a^{-1}a^{-1}=0\text{.}
\end{equation*}%
Hence $a^{-1}$ is also precise.
\end{proof}

\begin{proposition}
\label{Prop dist a(x+y)}Let $S$ be a solid, let $x,y\in S$ and let $a\in S$
be precise. Then $a\left( x+y\right) =ax+ay$.
\end{proposition}

\begin{proof}
From Axiom \ref{Axiom distributivity} and Proposition \ref{x0=0} we derive 
\begin{align*}
ax+ay& =a\left( x+y\right) +e(a)x+e(a)y \\
& =a\left( x+y\right) +0x+0y= \\
& =a\left( x+y\right) .
\end{align*}
\end{proof}

\begin{theorem}
The set of precise elements $P$ of a solid $S$ is an ordered field.
\end{theorem}

\begin{proof}
By Proposition \ref{precise closed}, precise elements are closed under
addition and multiplication. By Proposition \ref{prop zero}.\ref{1=note(1)}
the precise element $0$ is different from the precise element $1$. Clearly $%
\left( P,+\right) $ and $\left( P\backslash \left \{ 0\right \} ,\cdot
\right) $ are ordered abelian groups. The distributive property holds by
Proposition \ref{Prop dist a(x+y)}. Hence $P$ is an ordered field.
\end{proof}

\begin{corollary}
\label{u(a)=1}Let $S$ be a solid and let $a\in S$ be a non-zero precise
element. Then $u\left( a\right) =1$.
\end{corollary}

\section{Characterization of distributivity}

In solids addition is connected with multiplication via an adapted
distributivity condition. Distributivity holds up to a magnitude. Indeed, if 
$x,y$ and $z$ are elements of a solid, in general to obtain equality with $%
xy+xz$ one has to add to $x(y+z)$ a magnitude depending only on $y,z$ and
the magnitude of $x$. So we have subdistributivity with a concrete
expression of the correction term.

Theorem \ref{dist chacterization} gives necessary and sufficient conditions
for the usual distributive law to hold, i.e. under such conditions the
correction term is less than the magnitudes involved. It appears that the
only case where distributivity does not hold is the joint presence of two
circumstances: the factor $x$ should be more imprecise than the terms $y$
and $z$, and these terms should be almost opposite. That approximations may
not work out well in relation to differences of almost equal terms is
well-known, and may for instance be compared with the fact that if a
sequence of functions $f$ converges, the corresponding sequence of
derivatives (limits involving differences with almost equal terms of the
form $f(x+h)-f(x)$) does not need to converge.

\subsection{Conditions for distributivity\label{Subsection conditions
distributivity}}

We start by formulating necessary and sufficient conditions for
distributivity to hold. This requires a notion of relative size of magnitude
which corresponds to the notion given in \cite[p. 151]{koudjetivandenberg}
(see also \cite[Definition 3.2.11]{koudjetithese}).

\begin{definition}
Let $S$ be a solid and $x\in S$. The \emph{relative uncertainty of }$x$,
noted $R\left( x\right) $ is defined as follows. If $x\neq e\left( x\right) $%
, then $R(x)=e\left( u\left( x\right) \right) $.\ If $x=e\left( x\right) $,
then $R\left( x\right) =M$, where $M$ is given by Axiom \ref{Axiom neut max}.
\end{definition}

Observe that for zeroless $x$ one has $R(x)=e(x)x^{-1}$, so $R(x)$ expresses
the relative uncertainty indeed. Also $x=x(u(x)+R(x))$, for $%
x=xu(x)=x(u(x)+e(u(x))=x(u(x)+R(x))$.

For precise non-zero numbers the relative uncertainty is equal to $0$. In
case $x=e(x)$ is a magnitude the formula $R(x)=e(x)x^{-1}$ would amount to $%
R(x)=e(x)(e(x))^{-1}$, a division by a generalized zero. So it is natural to
define the relative uncertainty of a magnitude to be the largest magnitude.

\begin{theorem}
\emph{(Distributivity criterion)}\label{dist chacterization} Let $S$ be a
solid and let $x,y,z\in S$. Then 
\begin{equation}
xy+xz=x\left( y+z\right) \Leftrightarrow e\left( x\right) \left( y+z\right)
=e\left( x\right) y+e\left( x\right) z\vee R\left( x\right) \leq R\left(
y\right) +R\left( z\right) .  \label{dist charac}
\end{equation}
\end{theorem}

By the criterion above, distributivity holds for $x$ if it is true for its
magnitude; the only case where it is maybe not true is where $y$ and $z$ are
almost opposite, i.e. $y+z$ is much smaller than both $y$ and $z$. This is
easily seen in the most extreme case. If $x$ is not precise and $y$ and $z$
are precise non-zero with $y=-z,$ we have $e\left( x\right) \left(
y+z\right) =0$ and $e\left( x\right) y+e\left( x\right) z=e\left( x\right)
y-e\left( x\right) y=e\left( x\right) y\neq 0$. If $y$ and $z$ are not
precise, distributivity may hold, provided that the relative uncertainty of $%
x$ is less than or equal to the maximum of the relative uncertainties of $y$
and $z$, one might say "sharpness cuts"; note that we have already proved in
Proposition \ref{Prop dist a(x+y)} that distributivity always holds if $x$
is precise and in Corollary \ref{Cor dist magnitudes}\ that $e\left(
x\right) \left( e\left( y\right) +e\left( z\right) \right) =e\left( x\right)
e\left( y\right) +e\left( x\right) e\left( z\right) $. So the criterion
formalizes what was said above: distributivity does not hold in the joint
presence of two circumstances: the factor $x$ should be more imprecise than
the terms $y$ and $z$, and these terms should be almost opposite.

The distributivity axiom implies that subdistributivity takes the form of an
inequality. This is in line with subdistributivity for interval calculus
holding with inclusion.

\begin{theorem}
\emph{(Subdistributivity)}\label{subdist} Let $S$ be a solid and let $%
x,y,z\in S$. Then $x\left( y+z\right) \leq xy+xz$.
\end{theorem}

\begin{proof}
Using Proposition \ref{prop zero}.\ref{e(x)maiorzero} one has%
\begin{equation*}
x\left( y+z\right) \leq x\left( y+z\right) +e\left( x\right) y+e\left(
x\right) z=xy+xz.
\end{equation*}
\end{proof}

We already saw that for distributivity not to hold the two elements $y$ and $%
z$ of the second member should be in some sense opposite. This can never be
if $y$ or $z$ are of the same sign. We show that distributivity effectively
holds in this case.

\begin{corollary}
\label{dist positivos}Let $S$ be a solid and let $x,y,z\in S$. If $y$ and $z$
are of the same sign, then $x\left( y+z\right) =xy+xz$.
\end{corollary}

\begin{proof}
By Theorem \ref{dist chacterization} we only need to show that $e\left(
x\right) \left( y+z\right) =e\left( x\right) y+e\left( x\right) z$. Suppose
firstly that $y$ and $z$ are both positive. By Theorem \ref{subdist} one has 
$e\left( x\right) \left( y+z\right) \leq e\left( x\right) y+e\left( x\right)
z$. To show that also $e\left( x\right) y+e\left( x\right) z\leq e\left(
x\right) \left( y+z\right) $, we assume without loss of generality that $%
y\leq z$. Then by Axiom \ref{Axiom Amplification}%
\begin{equation*}
e\left( x\right) y+e\left( x\right) z=e\left( x\right) z\leq e\left(
x\right) \left( z+e\left( y\right) \right) \leq e\left( x\right) \left(
y+z\right) \text{.}
\end{equation*}%
Hence $e\left( x\right) \left( y+z\right) =e\left( x\right) y+e\left(
x\right) z$.

If $y$ and $z$ are both negative, then $-y$ and $-z$ are positive, hence 
\begin{equation*}
x\left( y+z\right) =-x\left( -y-z\right) =-x(-y)-x(-z)=xy+xz.
\end{equation*}%
\bigskip
\end{proof}

\begin{corollary}
\label{x(y+y)=xy+xy}Let $S$ be a solid and let $x,y\in S$. Then $x\left(
y+y\right) =xy+xy$.
\end{corollary}

To prove Theorem \ref{dist chacterization} we distinguish three cases: (i) $%
x $ is not zeroless, (ii) $y$ or $z$ is not zeroless, (iii) $x,y,z$ are
zeroless. In the first case we may assume that $y$ and $z$ are zeroless
hence $R\left( y\right) +R\left( z\right) <R\left( x\right) $. Then the
criterion states that $e\left( x\right) \left( y+z\right) =e\left( x\right)
y+e\left( x\right) z$ is equivalent to itself. In Proposition \ref{dist n
non-zeroless} we give a more operational version of this criterion. In the
second case one always has $R\left( x\right) \leq R\left( y\right) +R\left(
z\right) $. Then we must show that $x\left( z+e\left( y\right) \right)
=xz+xe\left( y\right) $ always holds; notice that if this particular form of
distributivity holds, automatically $R\left( x\right) \leq R\left(
e(y\right) )+R\left( z\right) $, because $R\left( e(y\right) )=M$. We do so
in Proposition \ref{dist y magnitude} below. In the third case both
conditions of the criterion may happen. Its proof is rather involved and
needs some preliminary results where we take profit of the fact that all
elements are zeroless.

\begin{proposition}
\label{dist n non-zeroless}Let $S$ be a solid, $x\in S$ and $y,z\in S^{\ast
} $. Suppose that $x=e\left( x\right) $. Then $e(x)\left( y+z\right)
=e(x)y+e(x)z$ if and only if $e\left( x\right) \left( y+z\right) =e\left(
x\right) y$ or $e\left( x\right) \left( y+z\right) =e\left( x\right) z$.
\end{proposition}

\begin{proof}
The direct implication follows directly from Axiom \ref{Axiom escala}\ and
Proposition \ref{theoremassembly}.\ref{assembly0}. For the inverse
implication, assume that $e\left( x\right) \left( y+z\right) =e\left(
x\right) y$ or $e\left( x\right) \left( y+z\right) =e\left( x\right) z$. Now 
$e\left( x\right) \left( y+z\right) \neq e\left( x\right) y+e\left( x\right)
z$ is self-contradictory, for then $e\left( x\right) y=e\left( x\right) z$
by Lemma \ref{lemma e(x)y=e(x)z}, hence in both cases $e\left( x\right)
\left( y+z\right) =e\left( x\right) y+e\left( x\right) z$. Hence $e\left(
x\right) \left( y+z\right) =e\left( x\right) y+e\left( x\right) z$ indeed.
\end{proof}

\begin{proposition}
\label{dist y magnitude}Let $S$ be a solid, $x,y,z\in S$. Then $x\left(
z+e\left( y\right) \right) =xz+xe\left( y\right) $.
\end{proposition}

\begin{proof}
One has $e\left( x\left( z+e\left( y\right) \right) \right) =e\left(
x\right) \left( z+e\left( y\right) \right) +x\left( e\left( z\right)
+e\left( y\right) \right) $, by Axiom \ref{Axiom e(xy)=e(x)y+e(y)x}. Then $%
e\left( x\right) z+e\left( x\right) e\left( y\right) \leq e\left( x\right)
z+e\left( x\right) e\left( y\right) +xe\left( z\right) +xe\left( y\right)
=e\left( x\left( z+e\left( y\right) \right) \right) $, by Proposition \ref%
{dist n non-zeroless} and Proposition \ref{dist neutrices}. Hence $e\left(
x\left( z+e\left( y\right) \right) \right) =e\left( x\left( z+e\left(
y\right) \right) \right) +e\left( x\right) z+e\left( x\right) e\left(
y\right) $, by Proposition \ref{e(x)maiore(y)iff}. Then $x\left( z+e\left(
y\right) \right) =x\left( z+e\left( y\right) \right) +e\left( x\right)
z+e\left( x\right) e\left( y\right) =xz+xe\left( y\right) $, by Axiom \ref%
{Axiom distributivity}.
\end{proof}

We now turn to the third case which is in some sense the generic case. The
proof of the necessity part needs some calculatory properties of multiples
of magnitudes. We will see that in the presence of the conditions of the
criterion for distributivity the correction term $e\left( x\right) y+e\left(
x\right) z$ must be smaller than the magnitude of $x\left( y+z\right) $.
Combining this fact with subdistributivity will yield distributivity.

Next lemma is a form of cross-multiplication. It is stated as an
implication. The converse is much more involved and will be given in Lemma %
\ref{if...then R(x)<=R(y)} below.

\begin{lemma}
\label{Lemma R(x)<=R(y)}Let $S$ be a solid and let $x,y\in S^{\ast}$. If $%
R\left( x\right) \leq R\left( y\right) $ then $e\left( x\right) y\leq
e\left( y\right) x$.
\end{lemma}

\begin{proof}
Assume that $R\left( x\right) \leq R\left( y\right) $. Then $e\left(
x\right) x^{-1}=e\left( u\left( x\right) \right) \leq e\left( u\left(
y\right) \right) =e\left( y\right) y^{-1}$. By Proposition \ref{xe(y)=xe(z)} 
\begin{equation*}
e\left( x\right) u\left( x\right) y=e\left( x\right) x^{-1}xy\leq e\left(
y\right) y^{-1}xy=e\left( y\right) u\left( y\right) x.
\end{equation*}%
Hence $e\left( x\right) y\leq e\left( y\right) x$ by Theorem \ref%
{equivalences}.\ref{eu=xe(u)}.
\end{proof}

Lemma \ref{lemma e(x)y=e(x)z} expresses the fact that when distributivity
does not hold for the magnitude of $x$ then $y$ and $z$ must be roughly of
the same order of magnitude, i.e. $e\left( x\right) y=e\left( x\right) z$.

\begin{lemma}
\label{lemma e(x)y=e(x)z}Let $S$ be a solid and let $x,y,z\in S$. If $%
e\left( x\right) \left( y+z\right) \neq e\left( x\right) y+e\left( x\right)
z $, then $e\left( x\right) y=e\left( x\right) z.$
\end{lemma}

\begin{proof}
Assume $e\left( x\right) \left( y+z\right) \neq e\left( x\right) y+e\left(
x\right) z$. Then $e\left( x\right) \left( y+z\right) <e\left( x\right)
y+e\left( x\right) z$ by Theorem \ref{subdist}. We may assume, without loss
of generality that $\left \vert y\right \vert \leq \left \vert z\right \vert 
$. Then by Axiom \ref{Axiom Amplification} 
\begin{equation}
e\left( x\right) y=e\left( x\right) \left \vert y\right \vert \leq e\left(
x\right) \left \vert z\right \vert =e\left( x\right) z.  \label{e(x)y<=e(x)z}
\end{equation}%
As a consequence $e\left( x\right) y+e\left( x\right) z=e\left( x\right) z$
by Proposition \ref{e(x)maiore(y)iff}, and we conclude that 
\begin{equation}
e\left( x\right) \left( y+z\right) <e\left( x\right) z.
\label{e(x)zmaiore(x)(y+z)}
\end{equation}

In order to show that $e\left( x\right) z\leq e\left( x\right) y$, from now
on we assume that $z$ is positive; the case that $z$ is negative is
analogous. We prove that $y$ and $z$ are of opposite sign and "not too
different", in a sense we prove that $y<-z/2$, note that $\left \vert
y\right \vert \leq \left \vert z\right \vert $ implies the lower bound $%
-z+e(y)\leq y+e(z)$. The inequality $y+y<-z$ will be obtained by successive
approximations. We prove firstly that $y<e\left( y\right) $. Suppose towards
a contradiction that $e\left( y\right) \leq y$. Then by compatibility with
addition and Axiom \ref{Axiom Amplification} one has $e\left( x\right) z\leq
e\left( x\right) \left( z+e\left( y\right) \right) \leq e\left( x\right)
\left( z+y\right) $, in contradiction with (\ref{e(x)zmaiore(x)(y+z)}).
Hence $y<e\left( y\right) $ and $\left \vert y\right \vert =-y$.

Secondly we show that $e\left( z\right) <-y$. Suppose towards a
contradiction that $-y\leq e\left( z\right) $. Then 
\begin{equation*}
-y+e\left( z\right) \leq e(z)+e\left( z\right) =e\left( z\right) \text{.}
\end{equation*}
On the other hand, because $y<e\left( y\right) $ 
\begin{equation*}
e\left( z\right) \leq e\left( y\right) +e\left( z\right) \leq-y+e\left(
z\right) .
\end{equation*}
Hence $-y+e\left( z\right) =e\left( z\right) $. But then 
\begin{align*}
e\left( x\right) \left( y+z\right) & =e\left( x\right) \left( -y-z\right)
=e\left( x\right) \left( -y+e\left( z\right) -z\right) \\
& =e\left( x\right) \left( e\left( z\right) -z\right) =e\left( x\right)
\left( -z\right) =e\left( x\right) z,
\end{align*}
in contradiction with (\ref{e(x)zmaiore(x)(y+z)}). Hence $e\left( z\right)
<-y$.

We show now that $y+y<-z$. Suppose towards a contradiction that $-z\leq y+y$%
. Using Axiom \ref{Axiom Amplification}, Theorem \ref{subdist} and (\ref%
{e(x)zmaiore(x)(y+z)}) one has%
\begin{align*}
e\left( x\right) z& =e\left( x\right) \left( z+z-z\right) \leq e\left(
x\right) \left( z+z+y+y\right) \\
& \leq e\left( x\right) \left( y+z\right) +e\left( x\right) \left(
y+z\right) =e\left( x\right) \left( y+z\right) <e\left( x\right) z,
\end{align*}%
which is a contradiction. Hence $y+y<-z$, and also $z+e\left( y\right) \leq
-y-y+e\left( z\right) $.

To finish the proof, using the facts that $y+y<-z$ and $e\left( z\right) <-y$%
, Axiom \ref{Axiom Amplification} and Theorem \ref{subdist} we see that%
\begin{align*}
e\left( x\right) z& \leq e\left( x\right) \left( z+e\left( y\right) \right)
\leq e\left( x\right) \left( -y-y+e\left( z\right) \right) \\
& \leq e\left( x\right) \left( -y-y-y\right) =e\left( x\right) \left(
y+y+y\right) \\
& \leq e\left( x\right) y+e\left( x\right) y+e\left( x\right) y=e\left(
x\right) y.
\end{align*}%
Hence 
\begin{equation}
e\left( x\right) z\leq e\left( x\right) y.  \label{e(x)ylargerthane(x)z}
\end{equation}%
Combining (\ref{e(x)y<=e(x)z}) and (\ref{e(x)ylargerthane(x)z}) one derives
that $e\left( x\right) z=e\left( x\right) y$.
\end{proof}

\begin{proof}[Proof of the necessity of the condition for distributivity of
Theorem \protect \ref{dist chacterization}.]
We need\linebreak \ only to consider the zeroless case. Assume firstly that $%
e\left( x\right) \left( y+z\right) =e\left( x\right) y+e\left( x\right) z$.
One has%
\begin{align*}
xy+xz& =x\left( y+z\right) +e\left( x\right) y+e\left( x\right) z \\
& =x\left( y+z\right) +e\left( x\left( y+z\right) \right) +e\left( x\right)
\left( y+z\right) \\
& =x\left( y+z\right) +xe\left( y+z\right) +e\left( x\right) \left(
y+z\right) +e\left( x\right) \left( y+z\right) \\
& =x\left( y+z\right) +xe\left( y+z\right) +e\left( x\right) \left(
y+z\right) \\
& =x\left( y+z\right) +e\left( x\left( y+z\right) \right) =x\left(
y+z\right) .
\end{align*}%
Secondly, assume that $R\left( x\right) \leq R\left( y\right) +R\left(
z\right) $, where we may suppose that $e\left( x\right) \left( y+z\right)
<e\left( x\right) y+e\left( x\right) z$. Then $e\left( x\right) y=e\left(
x\right) z$ by Lemma \ref{lemma e(x)y=e(x)z} and $e\left( x\right) y\leq
e\left( y\right) x$ by Lemma \ref{Lemma R(x)<=R(y)}. Hence, using
Proposition \ref{dist neutrices} and Axiom \ref{Axiom e(xy)=e(x)y+e(y)x}%
\begin{align*}
x\left( y+z\right) & =x\left( y+z\right) +e\left( x\left( y+z\right) \right)
\\
& =x\left( y+z\right) +e\left( x\right) \left( y+z\right) +xe\left( y\right)
+xe\left( z\right) \\
& =x\left( y+z\right) +e\left( x\right) \left( y+z\right) +xe\left( y\right)
+xe\left( z\right) +e\left( x\right) y \\
& =x\left( y+z\right) +xe\left( y+z\right) +e\left( x\right) y \\
& =x\left( y+z\right) +e\left( x\right) y+e\left( x\right) z \\
& =xy+xz.
\end{align*}
\end{proof}

We now turn to the sufficiency of this condition for distributivity of
Theorem \ref{dist chacterization}. For the zeroless case, the converse of
Lemma \ref{Lemma R(x)<=R(y)} is needed. The proof of this converse uses two
important properties of the unity element. First, one has $e\left( x\right)
u\left( y\right) =e\left( x\right) $ for arbitrary $x,y$, generalizing
Theorem \ref{equivalences}.\ref{eu=xe(u)}, and second the decomposition $%
u\left( x\right) =1+e\left( u\left( x\right) \right) $.

\begin{remark}
\label{u(x)-1}By Axiom \ref{numexterno} there is a precise element $b$ such
that $u\left( x\right) -1=b+e\left( u\left( x\right) -1\right) =b+e\left(
u\left( x\right) \right) $. So one can write $u\left( x\right) =1+b+e\left(
u\left( x\right) \right) $ and we must show that $b$ can be chosen equal to $%
0$.
\end{remark}

We prove first some preliminary results. If an element $x$ is zeroless, the
absolute value of the precise part of its decomposition must be larger than
its magnitude. Also, dividing the magnitude of $x$ by $x$ itself is the same
as dividing the magnitude by any of its precise parts.

\begin{proposition}
\label{proposition e(x)/x=e(x)/a}Let $S$ be a solid. Let $x=a+e\left(
x\right) \in S^{\ast }$ be such that $e\left( a\right) =0$. Then
\end{proposition}

\begin{enumerate}
\item \label{e(x)<moda}$e\left( x\right) <\left \vert a\right \vert $.

\item \label{e(x)/a=e(x)/x}$e\left( x\right) a^{-1}=e\left( x\right)
x^{-1}=e\left( u\left( x\right) \right) $.
\end{enumerate}

\begin{proof}
\ref{e(x)<moda}. Assume firstly that $e\left( x\right) <x$. Suppose that $%
a\leq e\left( x\right) $. Then by compatibility with addition $x=a+e\left(
x\right) \leq e\left( x\right) +e\left( x\right) =e\left( x\right) $, which
is a contradiction. Then $0<a$ by Proposition \ref{order sum}.\ref{y<=x y
positive then x positive}. Hence $e\left( x\right) <a=\left \vert
a\right
\vert $.

Assume secondly that $x<e\left( x\right) $. Then $a+e\left( x\right) <e(x)$,
hence $e(x)<-a$ by compatibility with addition. By transitivity $0<-a$, so $%
-a=\left \vert a\right \vert $. Hence $e\left( x\right) <\left \vert
a\right
\vert $.

\ref{e(x)/a=e(x)/x}. By Part \ref{e(x)<moda} one has $0\leq e\left( x\right)
<\left \vert a\right \vert $. We start by proving that $e\left( x\right)
a^{-1}\leq e\left( u\left( x\right) \right) $. Assume firstly $0<a$. Now $%
a\leq a+e\left( x\right) =x$. Hence $1=aa^{-1}\leq xa^{-1}$. Then by Axiom %
\ref{Axiom Amplification} and Theorem \ref{equivalences}.\ref{eu=xe(u)}%
\begin{equation}
e\left( u\left( x\right) \right) =e\left( x\right) x^{-1}=e\left( x\right)
x^{-1}\cdot 1\leq e\left( x\right) u\left( x\right) a^{-1}=e\left( x\right)
a^{-1}.  \label{desig 1 e(x)/a=e(x)/x}
\end{equation}%
Secondly, assume $a<0$. Then $0<-a$ and $0<\left( -a\right) ^{-1}$ by
Proposition \ref{1/x is positive}. Moreover, $-a\leq -a+e\left( x\right) =-x$
by Proposition \ref{prop zero}.\ref{e(x)maiorzero}. Hence $%
1=(-a)(-a)^{-1}\leq (-x)(-a)^{-1}=xa^{-1}$. Then we derive (\ref{desig 1
e(x)/a=e(x)/x}) as above.

We prove now that $e\left( x\right) a^{-1}\leq e\left( u\left( x\right)
\right) $. Using Theorem \ref{equivalences}.\ref{eu=xe(u)}\ and Proposition %
\ref{u(a)=1} one derives

\begin{align*}
e\left( x\right) a^{-1}& =e\left( x\right) u\left( x\right) a^{-1}=e\left(
x\right) x^{-1}(a+e(x))a^{-1} \\
& =e\left( x\right) x^{-1}\left( 1+e\left( x\right) a^{-1}\right) .
\end{align*}%
Then by Part \ref{e(x)<moda}%
\begin{equation*}
e\left( x\right) a^{-1}\leq e\left( x\right) x^{-1}\left( 1+aa^{-1}\right)
=e\left( x\right) x^{-1}\left( 1+1\right) .
\end{equation*}%
Hence by Theorem \ref{subdist} 
\begin{equation}
e\left( x\right) a^{-1}\leq e\left( x\right) x^{-1}+x^{-1}e\left( x\right)
=x^{-1}e\left( x\right) =e\left( u\left( x\right) \right) .
\label{desig 2 e(x)/a=e(x)/x}
\end{equation}%
The result follows from (\ref{desig 1 e(x)/a=e(x)/x}) and (\ref{desig 2
e(x)/a=e(x)/x}).
\end{proof}

Using the previous proposition one shows that the element $b$ of Remark \ref%
{u(x)-1}\ has to be "small" in the sense that the absolute value of $b$ is
less than or equal to the magnitude of the unity of $x$.

\begin{lemma}
\label{abs(b)<=e(u)}Let $S$ be a solid. Let $x=a+e\left( x\right) \in
S^{\ast }$. Suppose $u\left( x\right) =1+b+e\left( u\left( x\right) \right) $%
. Then $\left \vert b\right \vert \leq e\left( u\left( x\right) \right) $.
\end{lemma}

\begin{proof}
Because $R\left( e\left( x\right) \right) =M=R\left( e\left( u\left(
x\right) \right) \right) $ and $R(a)=0$, one derives 
\begin{align*}
x& =xu\left( x\right) =\left( a+e\left( x\right) \right) \left( 1+b+e\left(
u\left( x\right) \right) \right) \\
& =(a+e\left( x\right) )(1+b)+\left( a+e\left( x\right) \right) e\left(
u\left( x\right) \right) \\
& =a\left( 1+b\right) +e\left( x\right) \left( 1+b\right) +xe\left( u\left(
x\right) \right) .
\end{align*}%
Then, because $a$ is precise,%
\begin{equation*}
a+e\left( x\right) =a+ab+e\left( x\right) (1+b)+e\left( x\right) .
\end{equation*}%
Hence 
\begin{equation*}
e\left( x\right) =ab+e\left( x\right) \left( 1+b\right) +e\left( x\right) .
\end{equation*}%
Now Proposition \ref{e(x)maiore(y)iff} implies that $e\left( x\right) \left(
1+b\right) \leq e\left( x\right) $, for 
\begin{align*}
e\left( x\right) & =e\left( e\left( x\right) \right) =e\left( ab\right)
+e\left( e\left( x\right) \left( 1+b\right) \right) +e\left( e\left(
x\right) \right) \\
& =e\left( x\right) +e\left( x\right) \left( 1+b\right) .
\end{align*}%
Hence $e\left( x\right) =ab+e\left( x\right) $. Then $\left \vert
ab\right
\vert \leq e\left( x\right) $. Hence $\left \vert b\right \vert
\leq e\left( x\right) a^{-1}=e\left( u\left( x\right) \right) $ by
Proposition \ref{proposition e(x)/x=e(x)/a}.\ref{e(x)/a=e(x)/x}.
\end{proof}

We now show that a unity can be written as the sum of the precise unity $1$
and the imprecision of the unity.

\begin{theorem}
\emph{(Expansion of unity)}\label{Proposition expansion}Let $S$ be a solid
and let $x\in S^{\ast }$. Then $u\left( x\right) =1+e\left( u\left( x\right)
\right) $.
\end{theorem}

\begin{proof}
By Lemma \ref{abs(b)<=e(u)}\ one may suppose $u\left( x\right) =1+b+e\left(
u\left( x\right) \right) $ with $\left \vert b\right \vert \leq e\left(
u\left( x\right) \right) $. This means that $b+e\left( u\left( x\right)
\right) =e\left( u\left( x\right) \right) $. Hence $u\left( x\right)
=1+e\left( u\left( x\right) \right) $.
\end{proof}

\begin{corollary}
\label{1>e(u(x))}Let $S$ be a solid and let $x\in S^{\ast }$. Then $e\left(
u\left( x\right) \right) <1$.
\end{corollary}

Ordered fields satisfy the property that $0<1$ and $1\cdot 0=0$. We
generalized these properties to $e(u(x))<u(x)$ and $e(x)u(x)=e(x)$. Within
solids stronger properties are valid. Indeed, if $x$ and $y$ are arbitrary
elements of a solid, then $e(u(x))<1$ and $e(x)u(y)=e(x)$.

\begin{theorem}
\label{e(x)u(y)=e(x)}Let $S$ be a solid. Let $x\in S$ and $y\in S^{\ast}$.
Then $e\left( x\right) u\left( y\right) =e\left( x\right) $.
\end{theorem}

\begin{proof}
By Theorem \ref{Proposition expansion}%
\begin{equation*}
e\left( x\right) u\left( y\right) =e\left( x\right) \left( 1+e\left( u\left(
y\right) \right) \right) \text{.}
\end{equation*}%
Then by Corollary \ref{Cor.x(y+e(y))=xy+xe(y)} 
\begin{equation*}
e\left( x\right) u\left( y\right) =e\left( x\right) 1+e\left( x\right)
e\left( u\left( y\right) \right) =e\left( x\right) +e\left( x\right) e\left(
u\left( y\right) \right) \text{.}
\end{equation*}%
By Corollary \ref{1>e(u(x))} and Axiom \ref{Axiom Amplification} one has $%
e\left( x\right) e\left( u\left( y\right) \right) \leq e\left( x\right) $.
Hence $e\left( x\right) u\left( y\right) =e\left( x\right) $ by Proposition %
\ref{e(x)maiore(y)iff}.
\end{proof}

We are now able to derive the converse to Lemma \ref{Lemma R(x)<=R(y)}.

\begin{lemma}
\label{if...then R(x)<=R(y)}Let $S$ be a solid and let $x,y\in S^{\ast}$. If 
$e\left( x\right) y\leq e\left( y\right) x$ then $R\left( x\right) \leq
R\left( y\right) $.
\end{lemma}

\begin{proof}
Suppose that $e\left( x\right) y\leq e\left( y\right) x$. Without
restriction of generality we may assume that $x$ and $y$ are positive. Then $%
e\left( x\right) u\left( y\right) x^{-1}\leq e\left( y\right) y^{-1}u\left(
x\right) $ by Axiom \ref{compat mult}. By Theorem \ref{e(x)u(y)=e(x)} one
has $e\left( x\right) x^{-1}\leq e\left( y\right) y^{-1}$. Hence $R\left(
x\right) \leq R\left( y\right) $.
\end{proof}

\bigskip

\begin{proof}[Proof of the sufficiency of the condition for distributivity
of Theorem \protect \ref{dist chacterization}]
We only need to consider the zeroless case. Assume that $xy+xz=x\left(
y+z\right) $. With this equality, applying cancellation to Axiom \ref{Axiom
distributivity} gives 
\begin{equation*}
e\left( x\left( y+z\right) \right) =e\left( x\left( y+z\right) \right)
+e\left( x\right) y+e\left( x\right) z.
\end{equation*}%
Then 
\begin{equation}
e\left( x\right) y+e\left( x\right) z\leq e\left( x\left( y+z\right) \right)
=e\left( x\right) \left( y+z\right) +xe\left( y\right) +xe\left( z\right) .
\label{desig distr}
\end{equation}%
We consider three cases: (i) $e\left( x\right) y+e\left( x\right) z\leq
e\left( x\right) \left( y+z\right) $, and if (i) does not hold, (ii) $%
e\left( x\right) y+e\left( x\right) z\leq xe\left( y\right) $ and (iii) $%
e\left( x\right) y+e\left( x\right) z\leq xe\left( z\right) .$

(i) One has $e\left( x\right) y+e\left( x\right) z=e\left( x\right) \left(
y+z\right) $ by (\ref{desig distr}) and Theorem \ref{subdist}.

(ii) Lemma \ref{lemma e(x)y=e(x)z} implies that $e\left( x\right) y=e\left(
x\right) z$. Then $e\left( x\right) y\leq xe\left( y\right) $. Hence $%
R\left( x\right) \leq R\left( y\right) \leq R\left( y\right) +R\left(
z\right) $ by Lemma \ref{if...then R(x)<=R(y)}.

(iii) This case is analogous to case (ii).
\end{proof}

This completes the proof of Theorem \ref{dist chacterization}.

\subsection{Equivalent form of the distributivity axiom\label{Subsection
equivalent form}}

We show that the distributivity condition of Axiom \ref{Axiom distributivity}
and formula (\ref{dist charac}) of Theorem \ref{dist chacterization} are
equivalent. In order to do so we need a property of subdistributivity as
well as some special cases of the distributive law. All these properties are
supposed to be shown without using Axiom \ref{Axiom distributivity}, but in
the presence of (\ref{dist charac}).

\begin{proposition}
\label{lemma p eqv axioms}Let $x,y,z\in S$. Suppose (\ref{dist charac})
holds. Then $e\left( x\right) \left( y+z\right) \leq e\left( x\right)
y+e\left( x\right) z$.
\end{proposition}

The proposition is equivalent to Axiom \ref{Axiom distributivity} in case $%
x=e(x)$, for then it takes the form 
\begin{equation*}
e\left( x\right) y+e\left( x\right) z=e\left( x\right) \left( y+z\right)
+e\left( x\right) y+e\left( x\right) z\text{.}
\end{equation*}%
We now present the special cases of the distributive law needed to prove the
proposition.

\begin{lemma}
\label{Lemma precise}Let $x,y,z\in S$. If $x$ is precise, then $x\left(
y+z\right) =xy+xz$.
\end{lemma}

\begin{proof}
By Proposition \ref{x0=0} one has $e\left( x\right) \left( y+z\right)
=0\left( y+z\right) =0$ and $e\left( x\right) y+e\left( x\right) z=0+0=0$.
Hence $e\left( x\right) \left( y+z\right) =e\left( x\right) y+e\left(
x\right) z$ and one concludes that $x\left( y+z\right) =xy+xz$ by (\ref{dist
charac}).
\end{proof}

\begin{lemma}
\label{e(x)(p+p)}Let $x\in S$. Let $p\in S$ be precise. Then $e\left(
x\right) \left( p+p\right) =e\left( x\right) p$.
\end{lemma}

\begin{proof}
By Axiom \ref{Axiom Amplification} one has $e\left( x\right) \leq e\left(
x\right) \left( 1+1\right) $. Suppose that $e\left( x\right) <e\left(
x\right) (1+1)$. By Axiom \ref{scheiding neutrices} there exists a precise
element $q$ such that $e\left( x\right) <q<e\left( x\right) \left(
1+1\right) $. Then $e\left( x\right) <q/\left( 1+1\right) $, otherwise by
Lemma \ref{Lemma precise} one would have $q=q/\left( 1+1\right) +q/\left(
1+1\right) \leq e\left( x\right) +e\left( x\right) =e\left( x\right) $.
Hence $e\left( x\right) \left( 1+1\right) \leq q$, a contradiction. Then $%
e\left( x\right) \left( 1+1\right) \leq e\left( x\right) $ and one concludes
that $e\left( x\right) \left( 1+1\right) =e\left( x\right) $. Then $e\left(
x\right) \left( p+p\right) =e\left( x\right) p$ by Lemma \ref{Lemma precise}.
\end{proof}

\begin{lemma}
\label{Lemma e(x)(p+e(y))}Let $x,y\in S$. Let $p\in S$ be precise. Then $%
x\left( p+e\left( y\right) \right) =xp+xe\left( y\right) $.
\end{lemma}

\begin{proof}
Immediately from (\ref{dist charac}), for $R\left( x\right) \leq R\left(
p\right) +R\left( e\left( y\right) \right) =M$.
\end{proof}

\begin{proof}[Proof of Proposition \protect \ref{lemma p eqv axioms}]
Suppose without loss of generality that $\left \vert y\right \vert \leq
\left \vert z\right \vert $. If $y$ and $z$ have opposite signs then%
\begin{equation*}
e\left( x\right) \left( y+z\right) \leq e\left( x\right) z\leq e\left(
x\right) y+e\left( x\right) z.
\end{equation*}%
If $y$ and $z$ have the same sign we may assume that they are both positive
by Proposition \ref{es=ex}. Let $z=p+e\left( z\right) $. Then by Lemma \ref%
{Lemma e(x)(p+e(y))} and Lemma \ref{e(x)(p+p)} 
\begin{eqnarray*}
e\left( x\right) \left( y+z\right) &\leq &e\left( x\right) \left( z+z\right)
\\
&=&e\left( x\right) (\left( p+p\right) +e\left( z\right) ) \\
&=&e\left( x\right) \left( p+p\right) +e\left( x\right) e\left( z\right) \\
&=&e\left( x\right) p+e\left( x\right) e\left( z\right) =e\left( x\right)
z\leq e\left( x\right) y+e\left( x\right) z.
\end{eqnarray*}
\end{proof}

\begin{theorem}
Let $x,y,z\in S$. Then (\ref{dist charac}) holds if and only if $%
xy+xz=x\left( y+z\right) +e\left( x\right) y+e\left( x\right) z$.
\end{theorem}

\begin{proof}
By Theorem \ref{dist chacterization} we only need to prove the necessary
part. Suppose (\ref{dist charac}) holds. By Axiom \ref{numexterno} there is
a precise number $a$ such that $x=a+e\left( x\right) $. By Lemma \ref{Lemma
e(x)(p+e(y))}%
\begin{equation*}
x\left( y+z\right) +e\left( x\right) y+e\left( x\right) z=a\left( y+z\right)
+e\left( x\right) \left( y+z\right) +e\left( x\right) y+e\left( x\right) z.
\end{equation*}%
By Proposition \ref{lemma p eqv axioms} and Lemma \ref{Lemma precise} 
\begin{eqnarray*}
x\left( y+z\right) +e\left( x\right) y+e\left( x\right) z &=&a\left(
y+z\right) +e\left( x\right) y+e\left( x\right) z \\
&=&ay+az+e\left( x\right) y+e\left( x\right) z.
\end{eqnarray*}%
Then by Lemma \ref{Lemma e(x)(p+e(y))} 
\begin{equation*}
x\left( y+z\right) +e\left( x\right) y+e\left( x\right) z=\left( a+e\left(
x\right) \right) y+\left( a+e\left( x\right) \right) z=xy+xz\text{.}
\end{equation*}
\end{proof}

\appendix

\pagebreak

\begin{center}
\textsc{Appendix:\ List of axioms}
\end{center}

\label{Section Axioms}

\begin{enumerate}
\item \textbf{Axioms for addition}

\begin{axiom}
\label{assemblyassoc}$\forall x\forall y\forall z(x+\left( y+z\right)
=\left( x+y\right) +z).$
\end{axiom}

\begin{axiom}
\label{assemblycom}$\forall x\forall y(x+y=y+x).$
\end{axiom}

\begin{axiom}
\label{assemblyneut}$\forall x\exists e\left( x+e=x\wedge \forall f\left(
x+f=x\rightarrow e+f=e\right) \right) .$
\end{axiom}

\begin{axiom}
\label{assemblysim}$\forall x\exists s\left( x+s=e\left( x\right) \wedge
e\left( s\right) =e\left( x\right) \right) .$
\end{axiom}

\begin{axiom}
\label{assemblye(xy)}$\forall x\forall y\left( e\left( x+y\right) =e\left(
x\right) \vee e\left( x+y\right) =e\left( y\right) \right) .$
\end{axiom}

\item \textbf{Axioms for multiplication}

\begin{axiom}
\label{axiom assoc mult}$\forall x\forall y\forall z(x\left( yz\right)
=\left( xy\right) z).$
\end{axiom}

\begin{axiom}
\label{axiom com mult}$\forall x\forall y(xy=yx).$
\end{axiom}

\begin{axiom}
\label{axiom neut mult}$\forall x\neq e\left( x\right) \exists u\left(
xu=x\wedge \forall v\left( xv=x\rightarrow uv=u\right) \right) .$
\end{axiom}

\begin{axiom}
\label{axiom sym mult}$\forall x\neq e\left( x\right) \exists d\left(
xd=u\left( x\right) \wedge u\left( d\right) =u\left( x\right) \right) .$
\end{axiom}

\begin{axiom}
\label{axiom u(xy)}$\forall x\neq e\left( x\right) \forall y\neq e\left(
y\right) \left( u\left( xy\right) =u\left( x\right) \vee u\left( xy\right)
=u\left( y\right) \right) .$
\end{axiom}

\item \textbf{Order axioms}

\begin{axiom}
\label{(OA)reflex}$\forall x(x\leq x).$
\end{axiom}

\begin{axiom}
\label{(OA)antisym}$\forall x\forall y(x\leq y\wedge y\leq x\rightarrow
x=y). $
\end{axiom}

\begin{axiom}
\label{(OA)trans}$\forall x\forall y\forall z(x\leq y\wedge y\leq
z\rightarrow x\leq z).$
\end{axiom}

\begin{axiom}
\label{(OA)total}$\forall x\forall y(x\leq y\vee y\leq x).$
\end{axiom}

\begin{axiom}
\label{(OA)compoper}$\forall x\forall y\forall z\left( x\leq y\rightarrow
x+z\leq y+z\right) .$
\end{axiom}

\begin{axiom}
\label{Axiom e(x)maiory}$\forall x\forall y\left( y+e(x)=e(x)\rightarrow
\left( y\leq e(x)\wedge -y\leq e(x)\right) \right) .$
\end{axiom}

\begin{axiom}
\label{compat mult}$\forall x\forall y\forall z\left( \left( e\left(
x\right) <x\wedge y\leq z\right) \rightarrow xy\leq xz\right) .$
\end{axiom}

\begin{axiom}
\label{Axiom Amplification}$\forall x\forall y\forall z\left( \left( e\left(
y\right) \leq y\leq z\right) \rightarrow e\left( x\right) y\leq e\left(
x\right) z\right) .$
\end{axiom}

\item \textbf{Axioms relating addition and multiplication}

\begin{axiom}
\label{Axiom escala}$\forall x\forall y\exists z(e(x)y=e(z)).$
\end{axiom}

\begin{axiom}
\label{Axiom e(xy)=e(x)y+e(y)x}$\forall x\forall y\left(
e(xy)=e(x)y+e(y)x\right) .$
\end{axiom}

\begin{axiom}
\label{e(u(x))=e(x)d(x)}$\forall x\neq e(x)\left( e(u(x))=e(x)/x\right) .$
\end{axiom}

\begin{axiom}
\label{Axiom distributivity}$\forall x\forall y\forall z\left( xy+xz=x\left(
y+z\right) +e\left( x\right) y+e\left( x\right) z\right) .$
\end{axiom}

\begin{axiom}
\label{Axiom s(xy)=s(x)y}$\forall x\forall y\left( -(xy)=(-x)y\right) .$
\end{axiom}

\item \textbf{Axioms of existence}

\begin{axiom}
\label{Axiom neut min}$\exists m\forall x\left( m+x=x\right) .$
\end{axiom}

\begin{axiom}
\label{Axiom neut mult}$\exists u\forall x\left( ux=x\right) .$
\end{axiom}

\begin{axiom}
\label{Axiom neut max}$\exists M\forall x(e\left( x\right) +M=M).$
\end{axiom}

\begin{axiom}
\label{existencia neutrices}$\exists x\left( e\left( x\right) \neq 0\wedge
e\left( x\right) \neq M\right) .$
\end{axiom}

\begin{axiom}
\label{numexterno}$\forall x\exists a\left( x=a+e\left( x\right) \wedge
e\left( a\right) =0\right) .$
\end{axiom}

\begin{axiom}
\label{scheiding neutrices}$\forall x\forall y(x=e\left( x\right) \wedge
y=e(y)\wedge x<y\rightarrow \exists z(z\neq e(z)\wedge x<z<y)).$
\end{axiom}
\end{enumerate}

\end{document}